\documentclass[12pt]{article}
\usepackage{}
\usepackage{mathrsfs,graphicx,amsmath,amssymb,amsfonts}

\usepackage{appendix}
\usepackage{caption}
\usepackage{mathrsfs,graphicx,amsmath,amssymb,amsfonts,amsthm}
\usepackage{cite}
\usepackage{longtable}
\usepackage{supertabular}
\evensidemargin=\oddsidemargin\topmargin=-1.5cm

\newtheorem{lem}{Lemma}[section]

\newtheorem{thm}[lem]{Theorem}

\theoremstyle{definition}

\usepackage{tikz,xcolor,hyperref}
\definecolor{lime}{HTML}{A6CE39}
\DeclareRobustCommand{\orcidicon}{
	\begin{tikzpicture}
		\draw[lime, fill=lime] (0,0)
		circle[radius=0.16]
		node[white]{{\fontfamily{qag}\selectfont \tiny \.{I}D}};
	\end{tikzpicture}
	\hspace{-2mm}}
\foreach \x in {A, ..., Z}{%
	\expandafter\xdef\csname orcid\x\endcsname{\noexpand\href{https://orcid.org/\csname orcidauthor\x\endcsname}{\noexpand\orcidicon}}}

\begin{document}
\title{The spectral radius and the distance spectral radius of complements of block graphs
 }
\author{Xu Chen$^1$\hspace{-1.5mm}\orcidA{},
Dongjun Fan$^1$,
Rongxiao Shao$^1$,
Guoping Wang$^2$\footnote{Corresponding author. Email: xj.wgp@163.com.}\\
{\small 1. School of Statistics and Data Science, Xinjiang University of Finance and Economics,
}\\
{\small {\"U}r{\"u}mqi, Xinjiang 830012, P.R.China;}\\
{\small 2. School of Mathematical Sciences, Xinjiang Normal University, }\\
{\small {\"U}r{\"u}mqi, Xinjiang 830017, P.R.China}}

\date{}
\maketitle {\bf Abstract.}
In this paper, we determine the graphs whose spectral radius and distance spectral radius attain maximum and minimum among all complements of clique trees.
Furthermore, we also determine the graphs whose spectral radius and distance spectral radius attain minimum and maximum among all complements of block graphs, respectively.

{\flushleft{\bf MR(2020) Subject Classification:}} 05C12, 05C50, 05C69\\

\section{Introduction}

~~~~The {\it adjacency matrix} of $G$ is $A(G)=(a_{ij})_{n\times n}$,
where $a_{ij}=1$ if $v_i$ is adjacent to $v_j$,
and otherwise $a_{ij}=0$.
Since $A( G )$ is a real and symmetric matrix,
its eigenvalues can be arranged as
$\lambda_1(A(G))\ge \lambda_2(A(G))\ge \cdots \ge \lambda_n(A(G))$,
where eigenvalue $\lambda_1(A(G))$ is called the {\it spectral radius}.
Let $d_G( v_i,v_j ) $ be the least distance between $v_i$ and $v_j$ in $G$.
Then the {\it distance matrix} of $G$ is $D( G ) =( d_{ij} ) _{n\times n}$,
where $d_{ij}=d_G( v_i,v_j ) $.
Since $D( G )$ is an non-negative real symmetric matrix,
its eigenvalues can be arranged
$\lambda_1(D(G))\ge \lambda_2(D(G))\ge \cdots \ge \lambda_n(D(G))$,
where eigenvalue $\lambda_1(D(G))$
is called the {\it distance spectral radius}.
The {\it complement} of graph $G=( V( G ) ,E( G ) )$ is denoted by $G^c=( V( G^c ) ,E( G^c ) )$,
where $V( G^c ) =V( G ) $ and $E( G^c ) = \{ xy\notin E(G):x,y\in V( G )\} $.
The spectral radius and distance spectral radius of  complements of graphs have been studied,
see references \cite{F.Y.Z, J.G.S, L.S.C, Y.G.D, L.H.Q.2, Q.R, chenxu, chen.xu}.

Let $G$ be a connected simple graph.
A {\it cut vertex} of a connected graph $G$ is a vertex whose deletion results in a disconnected
graph.
A clique of a graph is a set of mutually adjacent vertices.
A {\it block} of $G$ is a maximal connected subgraph of $G$ that has no cut vertex.
If each block of graph is a clique,
then the graph is called {\it clique tree}.
In this paper, we determine the unique graphs whose spectral radius and distance spectral radius attain maximum and minimum among all complements of clique trees.
Furthermore, we also determine the unique graphs whose spectral radius and distance spectral radius
respectively attain minimum and maximum among all complements of block graphs.

\section {The spectral radius of complements of clique trees}

~~~~~~~Suppose $G$ is a connected simple graph with the vertex set
$V( G ) = \{ v_1,v_2,\cdots ,v_n  \} $.
If two vertices $u$ and $v$ are adjacent, then we write $uv$.
Let $x=( x_1,x_2,\cdots ,x_n ) ^T$, where $x_i$ corresponds to $v_i$, i.e., $x( v_i ) =x_i$
for $i=1,2,\cdots ,n$. Then
\begin{equation}
	x^TA(G)x=2\sum_{v_iv_j\in G}{x_ix_j}
\end{equation}

The neighbor $N_G(v)$ of the vertex $v$
of $G$ is the set of the vertices which are
adjacent to $v$.
Suppose that $x$ is an eigenvector of $A(G)$ corresponding
to the eigenvalue $\lambda$. Then for $v_i\in V (G)$, we have
\begin{equation}
	\lambda x_i=\sum_{v_j\in N_G(v_i)}x_j, ~for~ i=1,2,\cdots ,n.
\end{equation}

Suppose $C_T$ is a clique tree of order $n$ with $=s$ cliques
such that some two cut vertices are not adjacent.
If a clique $K$ of $C_T$ contains exactly one cut vertex $v$,
then we call $K$ the {\it end clique}, and call $v$ the {\it end cut vertex}.
Suppose $\widetilde{v}$ is another cut vertex of $C_T$. Let
\begin{equation}
	\begin{split}
		\widetilde{C}_T=C_T-\{vu|~u\in V(K)\}+\{\widetilde{v}u|~u\in V(K)\setminus{v}\}.
		\nonumber
	\end{split}
\end{equation}
We called the graph $\widetilde{C}_T$ obtained from $C_T$ by moving the clique $K$ from $v$ to $\widetilde{v}$.

We denote by $d(G)$ the diameter of $G$ which is farthest distance between all pairs of vertices. 

\begin{lem}
Let $C_T$ and $\widetilde{C}_T$ be two clique trees of order $n$.
Set $x=( x_1,x_2,\cdots ,x_n ) ^T$ to be a perron vector of $A(C_T^c)$ with respect to $\lambda_1(A(C_T^c))$.
If $x(v)\ge x(\widetilde{v})$, then
$\lambda_1(A(C_T^c))\le \lambda_1(A(\widetilde{C}_T^c))$ with equality if and only if $\widetilde{v}=v$.
\end{lem}

\begin{proof}
Since $C_T$ contains two cut vertices which are not adjacent, $d(C_T)>3$, and so $d(\widetilde{C}_T)\ge 3$.
Note that $A(G^c)+A(G)=J_n-I_n$.
Let $J_n$ be the matrix of order $n$ whose all entries are $1$, and let $I_n$ be the identity matrix of order $n$.
Recall $\lambda_1(A(C_T^c))$ is a spectral radius of $C_T^c$.
That is each entry of $x$ is positive.
Since $x(v)\ge x(\widetilde{v})$, by equation (1) we have
$$x^TA(C_T)x=2\sum_{v_iv_j\in C_T}{x_ix_j}\ge 2\sum_{v_iv_j\in \widetilde{C}_T}{x_ix_j}=x^TA(\widetilde{C}_T)x.$$
Then
\begin{equation}
	\begin{split}
		\lambda_1(A(C_T^c))&=x^TA(C_T^c)x\\
&=x^T(J_n-I_n)x-x^TA(C_T)x\\
&\le x^T(J_n-I_n)x-x^TA(\widetilde{C}_T)x\\
&=x^TA(\widetilde{C}_T^c)x
		\nonumber
	\end{split}
\end{equation}
By Rayleigh's theorem we have $x^TA(\widetilde{C}_T^c)x\le \lambda_1(A(\widetilde{C}_T^c))$.
Thus, we have  $\lambda_1(A(C_T^c))\le \lambda_1(A(\widetilde{C}_T^c))$.

We suppose for a contradiction that $\widetilde{v}\ne v$.
Note that $\lambda_1(A(\widetilde{C}_T^c))=\lambda_1(A(C_T^c))$. By equation (2) we have
\begin{equation}
	\begin{split}
		0&=|\lambda_1(A(\widetilde{C}_T^c))x(\widetilde{v})-\lambda_1(A(C_T^c))x(\widetilde{v})|\\
		&= |\big((A(\widetilde{C}_T)-A(C_T))x\big)(\widetilde{v})|.
		\nonumber
	\end{split}
\end{equation}
Whereas $|\big((A(\widetilde{C}_T)-A(C_T)x\big)(\widetilde{v})|=\sum_{u\in K\setminus{\{v\}}}x(u)> 0$.
This contradiction shows that the necessity holds. 
\end{proof}

Suppose $C_T$ contains two nonadjacent cut vertices.
Then $d(C_T)>3$.
Let $w$ and $w'$ be two cut vertices contained in the clique $K$.
Write the clique $K'$ containing $w'$.
Moving all cliques except $K'$ from their cut vertices to $w$ in $C_T$,
we get a graph isomorphic to the clique tree $\mathbb{S}(s-2,1)$ which exactly contains two end cut vertices $w$ and $w'$, and $s-1$ end cliques, as illustrated in Figure 1.
\begin{center}
	\begin{figure}[htbp]
		\centering
		\includegraphics[width=76mm]{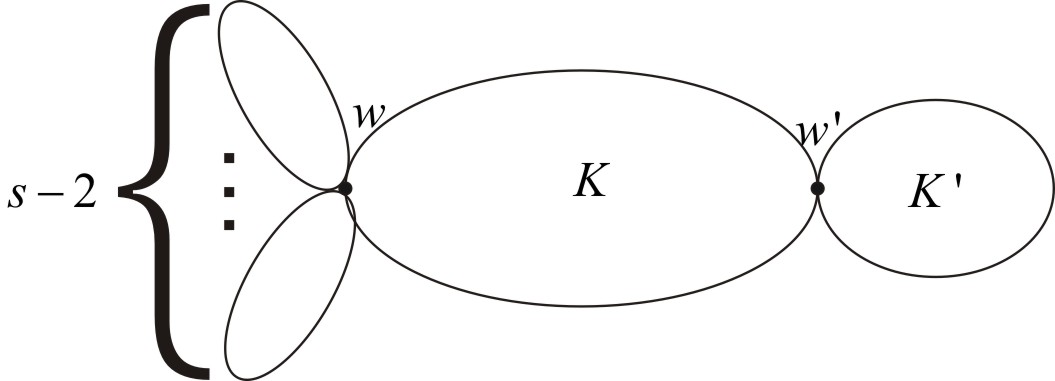}
		\\{{\bf Fig. 1.}  $\mathbb{S}(s-2,1)$.}
	\end{figure}
\end{center}

\begin{thm}
Suppose $C_T$ is a clique tree of order $n$ with $s$ cliques
such that some two cut vertices are not adjacent.
Set $x=( x_1,x_2,\cdots ,x_n ) ^T$ to be a perron vector of $A(C_T^c)$ with respect to $\lambda_1(A(C_T^c))$.
Then $\lambda_1(A(C_T^c))\le \lambda_1(A(\mathbb{S}^c(s-2,1)))$.
\end{thm}

\begin{proof} 
Let $x(w)$ be the minimum modulus among all cut vertices of $C_T$.
From the above construction of $\mathbb{S}(s-2,1)$ and the equation (1) we have
$$x^T A(C_T)x=\sum_{v_iv_j\in E(C_T)}x_ix_j\ge \sum_{v_iv_j\in E(\mathbb{S}(s-2,1))}x_ix_j=x^T A(\mathbb{S}(s-2,1))x.$$

Since $C_T$ contains two cut vertices which are not adjacent, $d(C_T)\geq 4$,
and so $d(C_T)>d(\mathbb{S}(s-2,1))=3$.
From Lemma 2.1 we have
\begin{equation}
	\begin{split}
		\lambda_1(A(C_T^c)) &=x^T(J_n-I_n)x-x^TA(C_T)x\\
		&\le x^T(J_n-I_n)x-x^TA(\mathbb{S}(s-2,1))x\\
		&= x^TA(\mathbb{S}^c(s-2,1)) x.
		\nonumber
	\end{split}
\end{equation}
By Rayleigh's theorem we have $\lambda _1(\mathbb{S}^c(s-2,1)) \ge x^TA(\mathbb{S}^c(s-2,1)) x$.
Thus, we have $\lambda_1(A(C_T^c))\le \lambda _1(\mathbb{S}^c(s-2,1))$. \end{proof} 

Let $\mathbb{C_T}_{n,d}$ denote the set of all clique trees of diameter $d$ with $s$ cliques on $n$ vertices. 

\begin{lem}
Let $d\ge 3$. Then
	$$\mathop {\max }\limits_{C_T\in \mathbb{C_T}_{n,d}}\lambda_1(A(C_T^c))\ge \mathop {\max }\limits_{C_T\in \mathbb{C_T}_{n,d+1}}\lambda_1(A(C_T^c)).$$
\end{lem}

\begin{proof}
Set $x=( x_1,x_2,\cdots ,x_n ) ^T$ to be a perron vector of $D(C_T^c)$ with respect to $\lambda_1(C_T^c))$.
Let $x(v)\ge x(\widetilde{v})$,
and move the clique $K$ from cut vertex $v$ to $\widetilde{v}$, 
and then move some one clique from their cut vertex $v'$ ($x(v')\ge x(\widetilde{v})$) to $\widetilde{v}$ in $C_T$,
we get a graph isomorphic to the clique tree $\widetilde{C}_T$.
Continue the above process in $\widetilde{C}_T$, we until get a graph $C_T'$ whose diameter less than $d$.
By applying Lemma 2.1 we obtain that the maximum modulus of $\lambda_1(A(C_T^c))$ is less than the maximum modulus of $\lambda_1(A(C_T'^c))$.
Thus, the result is clear. \end{proof}

$\mathbb{P}_{n_1,n_2,\cdots ,n_s}$ is called a {\it clique path} \cite{L.H.Q.2} if each edge of the path $P_{s+1}$  of order $s+1$ is replaced by a clique $K^i$
such that $V(K^i)\cap V(K^{i+1})=v_i$ for $i=1,2,\cdots, s-1$ and $V(K^i)\cap V(K^j)=\emptyset$ for $j\ne i-1, i+1$ and $2\le i\le s-1$.
If two graphs $G$ and $H$ are isomorphic, then we
write $G\cong H$.
Applying repeatedly Lemma 2.3 we have the following theorem. 

\begin{thm}
Suppose $C_T$ is a clique tree of order $n$ with $s$ cliques such that some two cut vertices are not adjacent. Then
$$\lambda_1(A(C_T^c))\ge \lambda_1(A(\mathbb{P}^c_{n_1,n_2,\cdots ,n_s}))$$
with equality if and only if $C_T\cong \mathbb{P}_{n_1,n_2,\cdots ,n_s}$.
\end{thm}

We denote by $T(n-3,1)$ the tree obtained from the path $P_3$ by appending $n-3$ vertices to some one end of $P_3$.
The following result are two special cases of Theorems 2.2 and 2.4.

\begin{thm}
Suppose $T$ is a tree of order $n$ with $d(T)>3$.
Then we have
$$\lambda_1(A(P_n^c))\le \lambda_1(A(T^c))\le \lambda_1(A(T^c(n-3,1))).$$
The equality holds if and only if $T\cong P_n$ and $T\cong T(n-3,1)$.
\end{thm} 

\begin{proof} Note that $T$ has exactly $n-1$ cliques.
Thus, by Theorems 2.2 and 2.4 we obtain that $\mathbb{P}_{n_1,n_2,\cdot,n_{n-1}}$ and $\mathbb{S}(n-2,1)$ are respectively $P_n$ and $T(n-3,1)$ such that $s=n-1$.
Then the result is clear. \end{proof} 

\section{The spectral radius of the complements of block graphs}

~~~~Let $B$ be a block graph of order $n$ with blocks $B^1, B^2, \cdots, B^s$.
Replacing each block $B^i$ of $B$ by clique $K^i$ of order $|V(K^i)|$, we get a graph isomorphic to the clique tree $C_B$. 
We denote by $\mathbb{B}_{n,d}$ the set of all block graphs of diameter $d$ with $s$ blocks on $n$ vertices.
Set $x=( x_1,x_2,\cdots ,x_n ) ^T$ to be a perron vector of $A(B^c)$ with respect to $\lambda_1(A(B^c))$. 

\begin{lem}
Suppose $B$ is a block graph of order $n$ with $s$ blocks whose diameter $d\ge 3$. Then
$$\mathop {\max }\limits_{B\in \mathbb{B}_{n,d+1}}\lambda_1(A(B^c))\ge \mathop {\max }\limits_{B\in \mathbb{B}_{n,d}}\lambda_1(A(B^c)).$$
\end{lem}

\begin{proof}
Deleting  some edges of some one block in $B$, and then connecting the above process, we will get a new block graph $B'$ whose diameter greater than $d$.
Set $x=( x_1,x_2,\cdots ,x_n ) ^T$ to be a perron vector of $A(B^c)$ with respect to $\lambda_1(A(B^c))$.
From the equation (1) we have
$$x^TA(B)x=2\sum_{v_iv_j\in B}{x_ix_j}\ge 2\sum_{v_iv_j\in B'}{x_ix_j}=x^TA(B')x.$$
Then
\begin{equation}
	\begin{split}
		\lambda_1(A(B^c)) &=x^T(J_n-I_n)x-x^TA(B)x\\
		&\le x^T(J_n-I_n)x-x^TA(B')x\\
		&= x^TA(B'^c) x.
		\nonumber
	\end{split}
\end{equation}
By Rayleigh's theorem we have $\lambda _1(A(B'^c)) \ge x^TA(B'^c)x$, and so $\lambda_1(A(B^c))\le \lambda_1(A(B'^c))$.
Thus, the result is clear.	
\end{proof}

Connecting  all pairs of vertices of each block which are not adjacent in $B$, and then applying Lemma 3.1, we will get a graph isomorphic to $C_B$.
Then we have the following result. 

\begin{lem}
Let $B$ and $C_B$ be two graphs of order $n$. 
Set $x=( x_1,x_2,\cdots ,x_n ) ^T$ to be a perron vector of $A(B^c)$ with respect to $\lambda_1(A(B^c))$.
If $B$ contains some two cut vertices which do not belong to the same block, then
$$\lambda_1(A(B^c)) \ge \lambda_1(A(C_B^c)).$$
The equality holds if and only if $B\cong C_B$.
\end{lem}

Combining Theorem 2.4 and Lemma 3.2 we have the following result. 

\begin{thm}
Suppose $B$ is a block graph of order $n$ with $s$ blocks such that some two cut vertices do not belong to the same block. Then
$$\lambda_1(A(B^c))\ge \lambda_1(A(\mathbb{P}^c_{n_1,n_2,\cdots ,n_s})).$$
The equality holds if and only if $B\cong \mathbb{P}^c_{n_1,n_2,\cdots ,n_s}$.
\end{thm}

\section{The distance spectral radius of the complements of clique trees}

~~~~Suppose the matrices $A=(a_{ij})_{n\times n}$ and $B=(b_{ij})_{n\times n}$.
Then we write $A=B$ if $a_{ij}=b_{ij}$, and $A\ge B$ if $a_{ij}\ge b_{ij}$.
The below Lemma 4.1 reflects the relationship of $D(G^c)$ and $A(G)$.   

\begin{lem}
[\citenum{chenxu}, Lemma 2.1]
Suppose $G$ is a simple graph on $n$ vertices whose diameter $d(G)$ is greater than two.
Then we have
\begin{enumerate}
\setlength{\parskip}{0ex}
\item[\rm (I.)]
when $d( G ) > 3$, $D( G^c ) =J_n-I_n+A( G )$.
\item[\rm (II.)]
when $d( G ) =3$, $D( G^c )\geq J_n-I_n+A( G )$.
\end{enumerate} 
\end{lem}

\begin{lem}
Let $C_T$ and $\widetilde{C}_T$ be two clique trees of order $n$.
Set $x=( x_1,x_2,\cdots ,x_n ) ^T$ to be a perron vector of $D(C_T^c)$ with respect to $\lambda_1(D(C_T^c))$.
If $x(\widetilde{v})\ge x(v)$, then
$\lambda_1(D(C_T^c))\le \lambda_1(D(\widetilde{C}_T^c))$ with equality if and only if $\widetilde{v}=v$.
\end{lem}

\begin{proof}
Since $C_T$ contains two cut vertices which are not adjacent, $d(C_T)>3$, and so $d(\widetilde{C}_T)\ge 3$.
By Lemma 4.1 and equation (1) we have
\begin{equation}
	\begin{split}
		\lambda_1(D(\widetilde{C}_T^c))-\lambda_1(D(C_T^c))&\ge x^T(D(\widetilde{C}_T^c)-D(C_T^c))x\\
		&= x^T(A(\widetilde{C}_T)-A(C_T))x\\
		&=2(x(\widetilde{v})-x(v))\sum_{u\in K\setminus{\{v\}}} x(u)\\
		&\ge 0
		\nonumber
	\end{split}
\end{equation}
Thus, we have  $\lambda_1(D(C_T^c))\le \lambda_1(D(\widetilde{C}_T^c))$.

We suppose for a contradiction that $\widetilde{v}\ne v$.
Note that $\lambda_1(D(\widetilde{C}_T^c))=\lambda_1(D(C_T^c))$. 
From equation (2) we have
\begin{equation}
	\begin{split}
		0&=|\lambda_1(D(\widetilde{C}_T^c))x(\widetilde{v})-\lambda_1(D(C_T^c))x(\widetilde{v})|\\
		&= |\big((A(\widetilde{C}_T)-A(C_T))x\big)(\widetilde{v})|.
		\nonumber
	\end{split}
\end{equation}
Whereas $|\big((A(\widetilde{C}_T)-A(C_T)x\big)(\widetilde{v})|=\sum_{u\in K\setminus{\{v\}}}x(u)> 0$.
This contradiction shows that the necessity holds. 
\end{proof} 

The proof is similar to the proof of Lemma 2.4.
Applying repeatedly Lemma 4.2 we have the following result.

\begin{lem}
Let $d\ge 3$. Then
$$\mathop {\max }\limits_{C_T\in \mathbb{C_T}_{n,d}}\lambda_1(D(C_T^c))\ge 
\mathop {\max }\limits_{C_T\in \mathbb{C_T}_{n,d+1}}\lambda_1(D(C_T^c)).$$
\end{lem}

Applying repeatedly Lemma 4.3 we have the following result.

\begin{lem}
Suppose $C_T$ is a clique tree of order $n$ with $s$ cliques such that some two cut vertices are not adjacent. 
Then $$\lambda_1(D(C_T^c))\ge \lambda_1(D(\mathbb{P}^c_{n_1,n_2,\cdots ,n_s}))$$
with equality if and only if $C_T\cong \mathbb{P}_{n_1,n_2,\cdots ,n_s}$.
\end{lem}

\begin{thm}
Suppose $C_T$ is a clique tree of order $n$ with $s$ cliques
such that some two cut vertices are not adjacent.
Set $x=( x_1,x_2,\cdots ,x_n ) ^T$ to be a perron vector of $D(C_T^c)$ with respect to $\lambda_1(D(C_T^c))$.
Then $\lambda_1(D(C_T^c))\le \lambda_1(D(\mathbb{S}^c(s-2,1)))$.
\end{thm}

\begin{proof} Let $x(w)$ be the maximum modulus among all cut vertices of $C_T$.
From the above construction of $\mathbb{S}(s-2,1)$ and the equation (1) we have
$$x^T A(C_T)x=\sum_{v_iv_j\in E(C_T)}x_ix_j\le \sum_{v_iv_j\in E(\mathbb{S}(s-2,1))}x_ix_j=x^T A(\mathbb{S}(s-2,1))x.$$

Since $C_T$ contains two cut vertices which are not adjacent, $d(C_T)\geq 4$,
and so $d(C_T)>d(\mathbb{S}(s-2,1))=3$.
From Lemma 4.1 we have
\begin{equation}
	\begin{split}
		\lambda_1(D(C_T^c)) &=x^TD(C_T^c)x\\
		&=x^T( J_n-I_n ) x+x^TA(C_T) x\\
		&\le x^T( J_n-I_n ) x+x^TA(\mathbb{S}(s-2,1))x\\
		&= x^TD(\mathbb{S}^c(s-2,1)) x.
		\nonumber
	\end{split}
\end{equation}
By Rayleigh's theorem we have $\lambda _1(\mathbb{S}^c(s-2,1)) \ge x^TD(\mathbb{S}^c(s-2,1)) x$.
Thus, $\lambda_1(D(C_T^c))\le \lambda _1(\mathbb{S}^c(s-2,1))$.
\end{proof} 

The proof is similar to the proof of Lemma 2.5, and the following result are two special cases of Theorems 4.4 and 4.5. 

\begin{thm}
Suppose $T$ is a tree of order $n$ with $d(T)>3$.
Then we have
$$\lambda_1(D(P^c_n))\le \lambda_1(D(T^c))\le \lambda_1(D(T^c(n-3,1))).$$
The equality holds if and only if $T\cong P_n$ and $T\cong T(n-3,1)$.
\end{thm}

\section{The distance spectral radius of the complements of block graphs}

\begin{lem}
Let $B$ and $C_B$ be two graphs of order $n$. 
Set $x=( x_1,x_2,\cdots ,x_n ) ^T$ to be a perron vector of $D(B^c)$ with respect to $\lambda_1(D(B^c))$.
If $B$ contains some two cut vertices which do not belong to the same block, then
$$\lambda_1(D(B^c)) \le \lambda_1(D(C_B^c)).$$
The equality holds if and only if $B\cong C_B$.
\end{lem}

\begin{proof} 
Connecting all pairs of vertices in $B^i$ ($i=1,2,\cdots ,n$) which are not adjacent in $B$, we get a graph isomorphic to $C_B$.
Obviously, from equation (1) we have $x^T A(B)x=\sum_{v_iv_j\in E(B)}x_ix_j\le \sum_{v_iv_j\in E(C_B)}x_ix_j=x^T A(C_B)x$.

Since $B$ contains some two cut vertices which do not belong to the same block, we have $d(B)\ge d(C_B)>3$.  From Lemma 4.1 we have
\begin{equation}
	\begin{split}
		\lambda_1(D(B^c)) &=x^TD(B^c)x\\
		&=x^T( J_n-I_n ) x+x^TA(B) x\\
		&\le x^T( J_n-I_n ) x+x^TA(C_B)x\\
		&\le x^TD(C_B^c) x.
		\nonumber
	\end{split}
\end{equation}
By Rayleigh's theorem we have $\lambda _1(C_B^c) \ge x^TD(C_B^c) x$.
Then $\lambda_1(D(B^c))\le \lambda_1(D(C_B^c))$.

Suppose for a contradiction that $B\not\cong C_B$.
Note that $\lambda_1(D(B^c))=\lambda_1(D(C_B^c))$.  Then we have
\begin{equation}
	\begin{split}
		0&=\lambda_1(D(C_B^c))-\lambda_1(D(B^c))\\
		&= x^T(A(C_B)-A(B))x\\
		&=\sum_{v_iv_j\in (E(C_B)-E(B))}x_ix_j.
		\nonumber
	\end{split}
\end{equation}
By hypothesis we have  $E(C_B)-E(B)\ne \emptyset$, and so $\sum_{v_iv_j\in (E(C_B)-E(B))}x_ix_j>0$.
Thus, the contradiction shows that the necessity holds. \end{proof} 

Combining Theorem 4.5 and Lemma 5.1 we get the following result.

\begin{thm}
Suppose $B$ is a block graph of order $n$ with $s$ blocks such that some two cut vertices do not belong to the same block. Then
$$\lambda_1(D(B^c))\le \lambda_1(D(\mathbb{S}^c(s-2,1))).$$
The equality holds if and only if $B\cong \mathbb{S}(s-2,1)$.
\end{thm}


\begin{thebibliography}{10}

\bibitem{J.G.S} G. Jiang, G. Yu, W. Sun, Z. Ruan.
The least eigenvalue of graphs whose complements have only two pendent vertices[J]. 
Applied Mathematics and Computation, 2018, 331: 112-119.
\bibitem{Y.G.D} G. Yu, Y. Fan, M. Ye,
The least signless Laplacian eigenvalue of the complements of unicyclic graphs[J]. 
Applied Mathematics and Computation, 2017, 306: 13-21.


\bibitem{L.H.Q.2} H. Lin, S. Drury,
The distance spectrum of complements of trees[J]. 
Linear Algebra and its Applications, 2017, 530: 185-201.





\bibitem{Q.R} R. Qin, D. Li, Y. Chen, J. Meng,
The distance eigenvalues of the complements of unicyclic graphs[J]. 
Linear Algebra and its Applications, 2020, 598: 49-67.




\bibitem{L.S.C} S. Li, S. Wang,
The least eigenvalue of the signless Laplacian of the complements of trees[J]. 
Linear Algebra and its applications, 2012, 436(7): 2398-2405.



\bibitem{chenxu} X. Chen, G. Wang. The distance spectrum of the complements of graphs of diameter greater than three[J]. Indian Journal of Pure and Applied Mathematics, 2022, https://doi.org/10.1007/s13226-022-00315-9.
\bibitem{chen.xu} X. Chen, G. Wang. The distance spectrum of the complements of graphs with two pendent vertices[J], Indian Journal of Pure and Applied Mathematics, 2022, https://doi.org/10.1007/s13226-022-00322-w.
\bibitem{F.Y.Z} Y. Fan, F. Zhang, Y. Wang, 
The least eigenvalue of the complements of trees[J]. 
Linear algebra and its applications, 2011, 435(9): 2150-2155.




\end{thebibliography}
\end{document}